\documentclass[11pt,reqno]{article}
\usepackage{amssymb}
\usepackage{amsmath}
\usepackage{amsthm}
\usepackage{amsfonts}
\usepackage{url}
\usepackage{hyperref}
\usepackage{breakurl}
\usepackage{comment}
\usepackage{graphicx}

\renewcommand{\geq}{\geqslant} 
\renewcommand{\leq}{\leqslant}
\renewcommand{\ge}{\geqslant}
\renewcommand{\le}{\leqslant}

\newcommand{\Lh}{\widehat{L}}

\theoremstyle{plain}
\newtheorem{theorem}{Theorem}
\newtheorem{corollary}{Corollary}
\newtheorem{lemma}{Lemma}

\theoremstyle{definition}

\newtheorem{remark}{Remark}

\newcommand{\dif}{{\,d}}	\newcommand{\e}{{\textrm e}}

\begin{document}
\bibliographystyle{plain}
\title{The mean square of the error term in the prime number theorem}
\author
{Richard P. Brent\footnote{Australian National University,
Canberra, Australia
{\tt <Pintz@rpbrent.com>}},\; 
David J. Platt\footnote{School of Mathematics, University of Bristol,
Bristol, UK 
{\tt <dave.platt@bris.ac.uk>}}\;
and Timothy S. Trudgian\footnote{School of Science, UNSW Canberra at ADFA, Australia 
{\tt <t.trudgian@adfa.edu.au>}}
}

\maketitle

\vspace*{-10pt}
\begin{abstract}
\vspace*{5pt}
\noindent
We show that, on the Riemann hypothesis, 
$\limsup_{X\to\infty}I(X)/X^{2} \le 0.8603$, where $I(X) = \int_X^{2X} 
(\psi(x)-x)^2\,dx.$
This proves (and improves on) a claim by Pintz from 1982. We also show unconditionally that  $\frac{1}{5\,374}\leq I(X)/X^2 $ for sufficiently large $X$, and that the $I(X)/X^{2}$ has no limit as $X\rightarrow\infty$.
\end{abstract}

\section{Introduction}
\noindent
Let $\psi(x) = \sum_{n\leq x} \Lambda(n)$ where $\Lambda(n)$ is the von Mangoldt function. By the prime number theorem we have $\psi(x) \sim x$. 
Littlewood (see \cite[Thm.\ 15.11]{MV}) showed that
$\psi(x)-x = \Omega_\pm(x^{1/2}\log\log\log x)$ as $x \to \infty$.
In view of Littlewood's result, 
it is of interest that, assuming 
the Riemann hypothesis (RH), the mean square
of $(\psi(x)-x)/x^{1/2}$ is bounded. Under RH we have 
\begin{equation}\label{cramer}
\psi(x) -x \ll x^{1/2} \log^{2} x, \quad \int_{X}^{2X} (\psi(x) -x)^{2}\dif x \ll X^{2}.
\end{equation}
Note that using the first bound in (\ref{cramer}) does not yield the second bound. Define 
\begin{equation}\label{turnip}
I(X):= \int_{X}^{2X} (\psi(x) - x)^{2}\dif x.
\end{equation} 
Unconditionally, it is known that $I(X)\gg X^{2}$.
Indeed Popov and Stechkin \cite[Thms.\ 6--7]{Sketch} showed that
\begin{equation}\label{pop}
\int_{X}^{2X} |\psi(x) - x|\dif x > \frac{X^{3/2}}{200},
\end{equation}
where $X$ is sufficiently large. On using Cauchy--Schwarz,
this shows that $I(X)/X^{2} \geq 1/(40\,000)$.

Pintz wrote a series of papers giving bounds on the constant in (\ref{pop}): \cite{Pintz3} has an ineffective constant, \cite[Cor.\ 1]{Pintz1} has $(22000)^{-1}$ and \cite[Cor.\ 1]{Pintz2} has $400^{-1}$. Under RH, Cram\'{e}r \cite{Cramer} proved that $I(X)\leq c X^{2}$
for sufficiently large $X$. Pintz \cite{Pintz1,Pintz2} claims that one may take $c=1$ for all $X$ sufficiently large. We are unaware of a proof of this, or of any similar results in the literature.

It follows from the above discussion that there exist positive constants $A_{1}$ and $A_{2}$ for which $A_{1} \leq I(X)X^{-2} \leq A_{2}$, for sufficiently large $X$. Actually the upper bound is conditional on RH whereas the lower bound is unconditional. 
The purpose of this article is give what we believe to be the best known bounds on $A_{1}$ and $A_{2}$.
\begin{theorem}\label{chalk}
Assume the Riemann hypothesis and let $I(X)$ be defined in~$(\ref{turnip})$.
Then, for $X$ sufficiently large we have $\frac{1}{5\,374} \leq I(X)X^{-2} \leq 0.8603$.
\end{theorem}
Presumably, both bounds in Theorem \ref{chalk} could be improved. We computed $I(X)$ for $X$ at every integer $\in[1,10^{11}]$ and include two plots showing its short term behaviour as Figures \ref{fig:I_100} and \ref{fig:I_1e11}.

\begin{figure}[tbp]
\centering
\fbox{\includegraphics[width=0.9\linewidth]{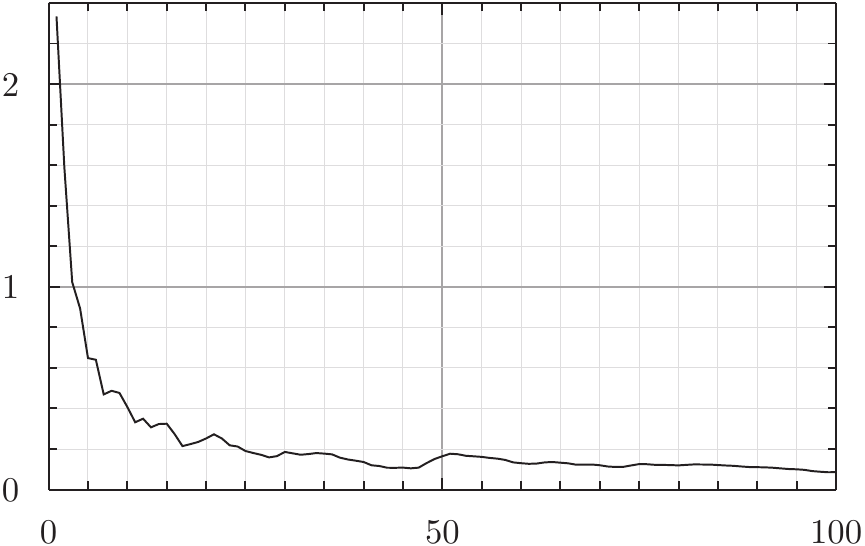}}
\caption{Plot of ${I(X)}/{X^2}$ vs $X$ for $X\in[1,100]$}
\label{fig:I_100}
\end{figure}

\begin{figure}[tbp]
\centering
\fbox{\includegraphics[width=0.9\linewidth]{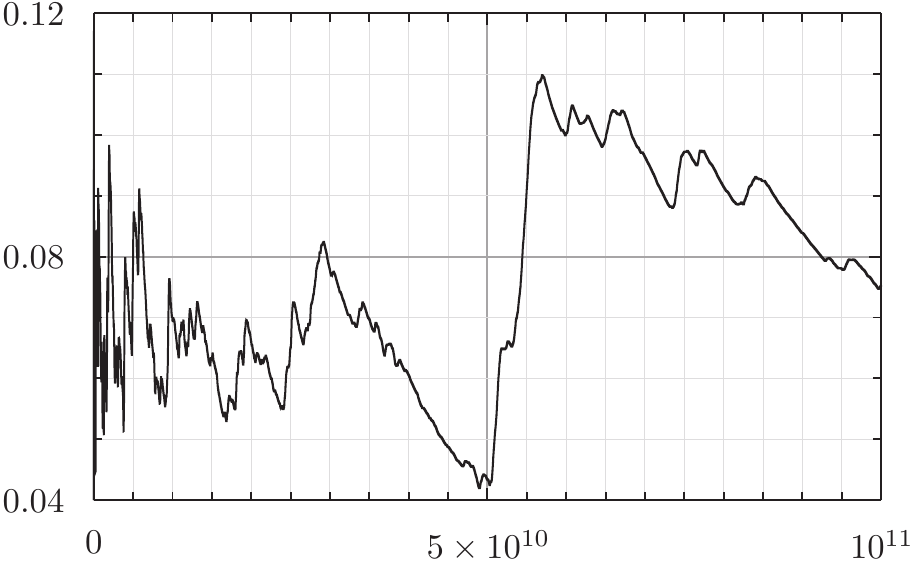}}
\caption{Plot of ${I(X)}/{X^2}$ vs $X$ sampled every $10^5$}
\label{fig:I_1e11}
\end{figure}

We are not aware of any conjectured results on the limiting behaviour of $I(x)x^{-2}$, and so prove the following.
\begin{theorem}\label{th:nolimit}
With $I(X)$ defined by $(\ref{turnip})$, we have that $\lim_{X\rightarrow\infty} I(X)/X^{2}$ does not exist.
\end{theorem}

If RH is false, then $I(X)/X^2$ is unbounded.
Hence, we assume RH except where noted 
(e.g.\ RH is not necessary in \S\ref{sec:lemmas}).
Let
\begin{equation}			\label{eq:c6}
B := \sum_{\rho_1,\rho_2} \left|\frac
{2^{2+i(\gamma_1-\gamma_2)}-1}
   {\rho_1{\overline{\rho_2}}(2+i(\gamma_1-\gamma_2))}\right|\,,
\end{equation}
where $\rho_j = \frac12+i\gamma_j$ denotes a nontrivial zero of $\zeta(s)$.
Following along the lines of \cite[Thm.\ 13.5]{MV}, one
can show that
\[
\limsup_{X\to\infty}\frac{I(X)}{X^2} \le B\,.
\]
Corollary \ref{cor:B_lower_upper_bounds} shows that $B \le 0.8603$. This proves the upper bound in Theorem \ref{chalk}, which proves Pintz's claim and provides a significant improvement.

In \S\ref{sec:lemmas} we give some variations on a well-known lemma of
Lehman that is useful for estimating bounds on sums over nontrivial zeros
of $\zeta(s)$. We then give several such bounds
that are used in the proof of Theorem~\ref{thm:main}.
In \S\ref{sec:bounding} we prove Theorem~\ref{thm:main}, which 
 bounds the tail of the sum in \eqref{eq:c6}, and 
in Corollary~\ref{cor:B_lower_upper_bounds} we deduce 
bounds on~$B$. 
In \S \ref{brutus} we prove the lower bound in Theorem \ref{chalk}. Finally, in \S \ref{parsnip} we prove Theorem \ref{th:nolimit}.

Throughout this paper we write $\vartheta$ to denote a complex number with modulus at most unity. Also, expressions such as $T/2\pi$ should be
interpreted as $T/(2\pi)$,
and $\log^k x$ as $(\log x)^k$.
The symbols $\gamma, \gamma_1, \gamma_2$ denote the ordinates of
generic nontrivial zeros $\beta + i\gamma$ of $\zeta(s)$. If we wish to
refer to the $k$-th such $\gamma > 0$ we denote it by
$\widehat{\gamma_k}$. For example, $\widehat{\gamma_1} = 14.13472514\cdots$. Finally, we define  $L = \log T$ and $\Lh = \log(T/2\pi)$.

\section{Preliminary results}			\label{sec:lemmas}

The results in this section are unconditional.

We state a well-known result due to Backlund~\cite{Backlund},
with the constants improved by several authors, most recently by
Trudgian~\cite[Thm.~1, Cor.~1]{TrudgianS2}, and
Platt and Trudgian~\cite[Cor.~1]{PTZ}.

\pagebreak[3]
\begin{lemma}[Backlund--Platt--Trudgian]		\label{lem:Backlund}
For all $T \ge 2\pi e$, 
\[
N(T) = \frac{T}{2\pi}\log\frac{T}{2\pi} - \frac{T}{2\pi} + \frac78
	+ Q(T),
\]
where
\[
|Q(T)| \le 0.11\log T + 0.29\log\log T + 2.29 + 0.2/T\,.
\]
\end{lemma}
On RH we have $Q(T) = O(\log T/\log\log T)$,
see \cite[Cor.\ 14.4]{MV}, but we do not use this result.

\begin{corollary}		\label{cor:improvedPTbound}
For all $T \ge 2\pi$,
\[
N(T) = \frac{T}{2\pi}\log\frac{T}{2\pi} - \frac{T}{2\pi} + \frac78
        + (0.28\vartheta)\log T.
\]
\end{corollary}
\begin{proof}
By Lemma~\ref{lem:Backlund},
the result holds for all $T \ge T_1 := 1.03 \cdot 10^8$. 
For $T \in [2\pi, T_1)$, 
it has been verified by an interval-arithmetic
computation, using the nontrivial zeros $\beta + i\gamma$ of $\zeta(s)$
with $\gamma \in (0, T_1)$. 
\end{proof}

Let $A$ be a constant such that
\begin{equation*}				
N(T) = \frac{T}{2\pi}\log\frac{T}{2\pi} - \frac{T}{2\pi} + \frac78
        + (\vartheta A)\log T
\end{equation*}
holds for all $T \ge 2\pi$. By Corollary~\ref{cor:improvedPTbound},
we can assume that $A \le 0.28$.

We state a lemma of Lehman~\cite[Lem.~1]{Lehman}.
We have generalised Lehman's wording,
but the original proof still applies.

\begin{lemma}[Lehman-decreasing]			\label{lem:Lehman}
If $2\pi e \le T_1 \le T_2$ and $\phi:[T_1,T_2]\mapsto [0,\infty)$
is monotone non-increasing on $[T_1,T_2]$, then
\[
\sum_{T_1<\gamma\le T_2}\!\!\phi(\gamma) =
\frac{1}{2\pi}\int_{T_1}^{T_2}\! \phi(t)\log(t/2\pi)\dif t
\,+\, A\vartheta\left(\!2\phi(T_1)\log T_1 + 
   \int_{T_1}^{T_2}\frac{\phi(t)}{t}\dif t\right)\!.
\]
\end{lemma}
In Lemma~\ref{lem:Lehman},
we can let $T_2 \to \infty$ if the first integral converges.
Lemma~\ref{lem:Lehman} does not apply if $\phi(t)$ is \emph{increasing}. 
In this case,
Lemma~\ref{lem:Lehman-rev} provides an alternative.

\begin{lemma}[Lehman-increasing]		\label{lem:Lehman-rev}
If $2\pi e \le T_1 \le T_2$ and $\phi:[T_1,T_2]\mapsto [0,\infty)$
is monotone non-decreasing on $[T_1,T_2]$, then
\[
\sum_{T_1<\gamma \le T_2}\!\!\phi(\gamma) =
 \frac{1}{2\pi}\int_{T_1}^{T_2}\! \phi(t)\log(t/2\pi)\dif t
 + A\vartheta\left(\!2\phi(T_2)\log T_2 + 
  \int_{T_1}^{T_2}\frac{\phi(t)}{t}\dif t\right)\!.
\]
\end{lemma}
\begin{proof}
We follow the proof of \cite[Lem.~1]{Lehman} with  appropriate
modifications. %
\end{proof}

We need to apply a Lehman-like lemma to a function $\phi(t)$ which 
decreases and then increases. Hence we state the following
lemma.
\begin{lemma}[Lehman-unimodal]			\label{lem:Lehman-combo}
Suppose that $2\pi e \le T_1 \le T_2$, and that
$\phi:[T_1,T_2]\mapsto[0,\infty)$.
If there exists $\theta\in[T_1,T_2]$ such that 
$\phi$ is non-increasing on $[T_1,\theta]$ and
non-decreasing on $[\theta,T_2]$,
then
\begin{align*}
\sum_{T_1 < \gamma \le T_2} \phi(\gamma)
 &= \frac{1}{2\pi}\int_{T_1}^{T_2}\phi(t)\log(t/2\pi)\dif t\\
 &\;\;\;\;+ A\vartheta\left(2\phi(T_1)\log T_1+2\phi(T_2)\log T_2
	+ \int_{T_1}^{T_2}\frac{\phi(t)}{t}\dif t\right).   
\end{align*}
\end{lemma}
\begin{proof}
Apply Lemma~\ref{lem:Lehman} on $[T_1,\theta]$ and 
Lemma~\ref{lem:Lehman-rev} on $[\theta,T_2]$.
\end{proof}

We need some elementary integrals.
For $k \ge 0$, $T \ge 1$ let
\begin{equation*}		
I_k := T\int_T^\infty \frac{\log^k t}{t^2}\dif t.
\end{equation*}
Then $I_0 = 1$ and $I_k$ satisfies the recurrence
$I_k = L^k + k I_{k-1}$ for $k \ge 1$.
Thus $I_1 = L+1$, 
$I_2 = L^2 + 2L + 2$,
$I_3 = L^3 + 3L^2 + 6L + 6$, etc.

We also need
\begin{equation}		\label{eq:L1t3}
T^2\int_T^\infty \frac{\log t}{t^3}\dif t = \frac{2L+1}{4}
\end{equation}
and
\begin{equation}		\label{eq:L2t3}
T^2\int_T^\infty \frac{\log^2 t}{t^3}\dif t = \frac{2L^2+2L+1}{4}\,,
\end{equation}
which may be found in a similar fashion to 
$I_1$ and $I_2$ respectively.

We now state some lemmas that will be used in \S\ref{sec:bounding}.
Lemmas \ref{lem:sum0bd}--\ref{lem:recip_gamma} 
are applications of Lemma~\ref{lem:Lehman}.

\begin{lemma}				\label{lem:sum0bd}
If $T \ge 2\pi e$, %
then
\[
\sum_{\gamma > T} \frac{1}{\gamma^2}
 \le \frac{L}{2\pi T}.  
\]
\end{lemma}
\begin{proof}
We apply Lemma~\ref{lem:Lehman} with $\phi(t) = 1/t^2$, 
\hbox{$T_1 = T$}, and
let the upper limit $T_2 \to \infty$.
Using the integral $I_1$ above, this gives
\begin{align*}
\sum_{\gamma > T} \frac{1}{\gamma^2} 
&= \frac{1}{2\pi}\int_T^\infty \frac{\log(t/2\pi)}{t^2}\dif t
  + A\vartheta\left(\frac{2L}{T^2} + \int_T^\infty\frac{dt}{t^3}\right)\\
&= \frac{L + 1 - \log(2\pi)}{2\pi T} + 
    A\vartheta\left(\frac{4L+1}{2T^2}\right)\\
&\le \frac{L}{2\pi T}\,,
\end{align*}
where the final inequality uses $T \ge 2\pi e$ and $A \le 0.28$.
\end{proof}

\begin{lemma}				\label{lem:sum1bd}
If $T \ge 4\pi e$, then
\[
\sum_{\gamma > T} \frac{\log(\gamma/2\pi)}{\gamma^2}
 \le \frac{L^2-L}{2\pi T}\,.
\]
\end{lemma}
\begin{proof}
We apply Lemma~\ref{lem:Lehman} with $\phi(t) = \log(t/2\pi)/t^2$,
\hbox{$T_1 = T$}, and let the upper limit $T_2 \to \infty$.
Since $\log(t/2\pi)/t^2$ is decreasing on $[4\pi e, \infty)$,
Lemma~\ref{lem:Lehman} is applicable. 
Making use of the integrals $I_2$ and~\eqref{eq:L1t3} 
above, we obtain
\begin{align*}
\sum_{\gamma > T} \frac{\log(\gamma/2\pi)}{\gamma^2}
 &= \frac{1}{2\pi}\int_T^\infty \frac{\log^2(t/2\pi)}{t^2}\dif t\\
    &\;\;\;\;+ A\vartheta\left(\frac{2\log(T/2\pi)\log T}{T^2}
	+\int_T^\infty\frac{\log(t/2\pi)}{t^3}\dif t\right)\\
 &= \frac{\Lh^2 + 2\Lh + 2}{2\pi T}
	+ A\vartheta\left(\frac{2L\Lh}{T^2}+\frac{2\Lh+1}{4T^2}\right) \le \frac{L^2-L}{2\pi T}\,,
\end{align*}
where the final inequality uses $T \ge 4\pi e$ and $A \le 0.28$.
\end{proof}

\begin{lemma}					\label{lem:sum2bd}
If $T \ge 100$,
then
\[
\sum_{\gamma > T} \frac{\log^2(\gamma/2\pi)}{\gamma^2}
 \le \frac{L^3 - 1.39L^2}{2\pi T}.
\]
\end{lemma}
\begin{proof}
We apply Lemma~\ref{lem:Lehman}
with $\phi(t) = \log^2(t/2\pi)/t^2$, $T_1=T$, and \hbox{$T_2\to\infty$}.
Since $\phi(t)$ 
is monotonic decreasing on $[100, \infty)$,
Lemma~\ref{lem:Lehman} is applicable.
\hbox{Using} the integrals $I_3$ and~\eqref{eq:L2t3} 
above, we obtain
\begin{align*}
\sum_{\gamma > T} \frac{\log^2(\gamma/2\pi)}{\gamma^2}
 &= \frac{1}{2\pi}\int_T^\infty \frac{\log^3(t/2\pi)}{t^2}\dif t\\
    &\;\;\;\;+ A\vartheta\left(\frac{2\log^2(T/2\pi)\log T}{T^2}
	+\int_T^\infty\frac{\log^2(t/2\pi)}{t^3}\dif t\right)\\
   &= \frac{\Lh^3 + 3\Lh^2 + 6\Lh + 6}{2\pi T}
     + A\vartheta\left(\frac{8L\Lh^2 + 2\Lh^2 + 2\Lh + 1}
	{4T^2}\right)\\
 &\le \frac{L^3-1.39L^2}{2\pi T}\,,
\end{align*}
where the final inequality uses $T \ge 100$ and $A \le 0.28$.
\end{proof}

The following lemma improves on the upper bound of~\cite[Lem.\ 2.10]{Saouter}.
\begin{lemma}				\label{lem:recip_gamma}
If $T \ge 4\pi e$, 		then
\begin{equation}				\label{eq:recipLh4pi}
\sum_{0 < \gamma \le T} \frac{1}{\gamma} 
\le \frac{\Lh^2}{4\pi}
\,.
\end{equation}
\end{lemma}
\begin{proof}		Suppose that $T \ge T_1$, where $T_1 \ge 4\pi e$ will be determined later.
Using Lemma~\ref{lem:Lehman} with $\phi(t) = 1/t$, we obtain
\begin{align}
\nonumber
\sum_{T_1 < \gamma \le T} \frac{1}{\gamma}
 &= \frac{1}{2\pi}\int_{T_1}^T \frac{\log(t/2\pi)}{t}\dif t
   + A\vartheta\left(\frac{2\log T_1}{T_1} + 
	\int_{T_1}^T\frac{dt}{t^2}\right)\\
\label{eq:harmonic-bound}
 &= \frac{1}{4\pi}\left(\Lh^2 - \log^2(T_1/2\pi)\right)
   + A\vartheta\left(\frac{2\log T_1+1}{T_1}\right)\,.
\end{align}
Thus, including a sum over $\gamma \le T_1$, we have
\begin{align*}
\sum_{0 < \gamma \le T} \frac{1}{\gamma}
 &\le \frac{\Lh^2}{4\pi} + \varepsilon(T_1),
\end{align*}
where
\begin{align*}
\varepsilon(T_1) = 
\sum_{0 < \gamma \le T_1}\frac{1}{\gamma}
 - \frac{\log^2(T_1/2\pi)}{4\pi} 
 + A\left(\frac{2\log T_1  + 1}{T_1}\right)\,.
\end{align*}
Using $A \le 0.28$,
and summing over the first $80$ nontrivial zeros of $\zeta(s)$,
shows that $\varepsilon(202) < 0$.
Thus, we take $T_1 = 202$, whence \eqref{eq:recipLh4pi} holds for $T \ge T_1 = 202$.
We can verify numerically that~\eqref{eq:recipLh4pi} 
also holds for $T \in [4\pi e,T_1)$.
\end{proof}

\begin{remark}				\label{rem:recip_gamma}
The motivation for our proof of Lemma~\ref{lem:recip_gamma} is
as follows.
Define
\[
H := \lim_{T\to\infty} \left(\sum_{0<\gamma\le T}\frac{1}{\gamma}
	- \frac{\log^2(T/2\pi)}{4\pi}\right)\,.
\]
It is easy to show, using~\eqref{eq:harmonic-bound},
that the limit defining $H$ exists.
A computation 
shows that $H \approx -0.0171594$.
Since $H$ is negative, we expect that $\varepsilon(T_1)$ should be
negative for all sufficiently large $T_1$. See also \cite{Hassani},
and \cite[Lem.~3]{Buethe}.
\end{remark}

\section{Bounding the tail in the series for $B$}	\label{sec:bounding}

We are now ready to bound the tail of the series~\eqref{eq:c6}.
Our main result is stated in Theorem~\ref{thm:main}.
Bounds on $B$ are deduced in Corollary~\ref{cor:B_lower_upper_bounds}.

\begin{theorem}				\label{thm:main}
Assume RH.
If $T \ge 100$, 
$L = \log T$, and $B$ is defined by~$\eqref{eq:c6}$, then
\[
B \le \sum_{|\gamma_1|\le T,|\gamma_2|\le T} \left|\frac
{2^{2+i(\gamma_1-\gamma_2)}-1}
   {\rho_1{\overline{\rho_2}}(2+i(\gamma_1-\gamma_2))}\right|
 + \frac{10L^3 + 11L^2}{\pi^2 T}\,.
\]
\end{theorem}
\begin{proof}
Initially, we ignore the numerators
$|2^{2+i(\gamma_1-\gamma_2)}-1|$ in~\eqref{eq:c6}, since they are
easily bounded.
Define 
\begin{equation}		 	\label{eq:ST}
S(T) := \sum_{|\gamma_1|\le T,|\gamma_2|\le T}\left|\frac
{1} 
{\rho_1{\overline{\rho_2}}(2+i(\gamma_1-\gamma_2))}\right|\,,
\end{equation}
and 
$S_\infty				
:= \lim_{T\to\infty}S(T)$, with $S_{\infty} \approx 0.217$. 
We refer to \hbox{$E(T) := S_\infty - S(T)$} 
as the \emph{tail} of the series with
{parameter $T$}. Thus, the tail is the sum of terms
with \hbox{$\max(|\gamma_1|,|\gamma_2|)>T$}.
Comparing with~\eqref{eq:c6}, and using
\hbox{$|2^{2+i(\gamma_1-\gamma_2)}-1| \le 5$},
we see that the error caused by summing \eqref{eq:c6}
with \hbox{$\max(|\gamma_1|,|\gamma_2)\le T$} is
at most $5E(T)$.

We consider bounding sums of the tail terms. By using the symmetry 
$(\gamma_1,\gamma_2) \to (-\gamma_1,-\gamma_2)$, i.e.\ complex conjugation,
we can assume that $\gamma_1 > 0$ (but we must multiply the
resulting bound by $2$).
We can also use the symmetry $(\gamma_1,\gamma_2)\to(\gamma_2,\gamma_1)$
if $\gamma_2 > 0$, and $(\gamma_1,\gamma_2)\to(-\gamma_2,-\gamma_1)$ if
$\gamma_2 < 0$, to reduce to the case that $|\gamma_2| \le \gamma_1$ 
(again doubling the resulting bound). Terms on the diagonal $\gamma_1=\gamma_2$ and
anti-diagonal $\gamma_1=-\gamma_2$ are given double the necessary
weight, but this does not affect the validity of the bound.

For each $\gamma_1 > 0$, possible $\gamma_2$ satisfy
$\gamma_2 \in [-\gamma_1, \gamma_1]$.  
Since $\gamma_2$ is the ordinate of a nontrivial
zero of $\zeta(s)$, it is never zero, in fact $|\gamma_2| > 14$.

We now bound the terms $1/|\rho_1\overline{\rho_2}(2+i(\gamma_1-\gamma_2))|$
and various sums.
Our strategy is to fix $\gamma_1$ and sum over all possible $\gamma_2$,
then allow $\gamma_1$ to vary and sum over all $\gamma_1 > T$.
Since $|\gamma_1| < |\rho_1|$ and $|\gamma_2| < |\rho_2|$, we actually bound
\[
t(\gamma_1,\gamma_2) 
 := \frac{1}{|\gamma_1\gamma_2(2+i(\gamma_1-\gamma_2))|}\,,
\]
which is only slightly larger, 
since $1 \le |\rho_j/\gamma_j| \le 1 + 1/8\gamma_j^2 \le 1.001$.

It is useful to define 
$D := 1/t(\gamma_1,\gamma_2)$.
We assume that $T \ge T_0 = 100$.
Since we eventually sum over $\gamma_1 > T$, we also
assume that $\gamma_1 \ge T_0$.

First suppose that $\gamma_2$ is positive.
In this case,
we have $0 < \gamma_2 \le \gamma_1$ and
$D \ge \gamma_1\gamma_2\max(2,\gamma_1-\gamma_2)$.
Thus the terms $t(\gamma_1,\gamma_2)$
are bounded by $\phi(\gamma_2)/\gamma_1^2$, where,
writing $T = \gamma_1$,
\begin{equation*}		
\phi(t) := 
\begin{cases}
\displaystyle
\frac{T}{t(T-t)} = \;
 \frac{1}{t} + \frac{1}{T-t}\; \text{ if } t\in (0,T-2];\\[11pt]
\displaystyle
\;\;
\frac{T/2}{T-2}\;\; = \;
 \frac{1}{2} + \frac{1}{T-2} \text{ if } t\in (T-2,T].
\end{cases}
\end{equation*}
Note that $\phi(t)$ is positive, decreasing on the interval
$(0, T/2]$, increasing on the interval $(T/2,T-2]$,
and constant on the interval $[T-2, T]$.
Thus, for summing $\phi(\gamma_2)$ over $\gamma_2\in(2\pi e, T]$,
Lemma~\ref{lem:Lehman-combo} applies with
\hbox{$T_1 = 2\pi e$}, $T_2 = T \ge 2T_1$, and $\theta = T/2$.

To apply Lemma~\ref{lem:Lehman-combo}, we need to bound
$(1/2\pi)\int_{T_1}^T \phi(t)\log(t/2\pi)\dif t$ (the main term),
and also the error terms
$A\int_{T_1}^T (\phi(t)/t)\dif t$
and $2A\phi(T_j)\log(T_j)$ ($j = 1, 2$).
We consider these in turn.

First consider the main term:
\begin{align*}
 &\frac{1}{2\pi}\int_{T_1}^T \phi(t)\log(t/2\pi)\dif t\\
=&\; \frac{1}{2\pi}\left(\int_{T_1}^{T-2} \left(\frac{1}{t} 
	+ \frac{1}{T-t}\right) \log(t/2\pi)\dif t 
	+ \phi(T)\int_{T-2}^T \log(t/2\pi)\dif t\right)\\
\le&\; \frac{1}{2\pi}\left(\int_{T_1}^T 
	  \frac{\log(t/2\pi)}{t}\dif t 
	  + \Lh \int_{0}^{T-2} \frac{dt}{T-t} 
	  + \Lh\left(1+\frac{2}{T-2}\right) \right)\\
\le&\; \frac{1}{4\pi}\left(\Lh^2-1 + 2\Lh\log(T/2) 
	  + 2\Lh + \frac{4\Lh}{T-2}\right)\\
\le&\; \frac{1}{4\pi}\left(3\Lh^2 + 2\Lh(2+\log \pi) - 0.88\right)\,.
\end{align*}

Now consider the error terms. We have
\begin{align*}
\int_{T_1}^T\frac{\phi(t)}{t}\dif t
 &= \int_{T_1}^{T-2}\frac{\phi(t)}{t}\dif t
  + \phi(T)\int_{T-2}^T \frac{dt}{t}\\
 &= \int_{T_1}^{T-2}\left(\frac{1}{t^2} + \frac{1}{T}\left(
	\frac{1}{t} + \frac{1}{T-t}\right)\right)\dif t
  + \phi(T)\int_{T-2}^T \frac{dt}{t}\\
 &\le \frac{1}{T_1} -\frac{1}{T} + \frac{\log(T/T_1) + \log(T/2)}{T}
	+ \frac{T}{(T-2)^2} \le 0.12\,.
\end{align*}
Also,
\[
2\phi(T_1)\log T_1 
 = \frac{2\log T_1}{T_1}\left(\frac{T}{T-T_1}\right) \le 0.41,
\]
and
\[
2\phi(T_2)\log T_2 \le \left(1+\frac{2}{T-2}\right)\log T
	\le \Lh + \log(2\pi) + \frac{2 \log T}{T-2}
	\le \Lh + 1.94\,.
\]
Thus, Lemma~\ref{lem:Lehman-combo} gives
\begin{align*}
\sum_{T_1 < \gamma \le T}\!\! \phi(\gamma)
 &\le \frac{3\Lh^2 + 2\Lh(2+\log\pi) - 0.88}{4\pi}
    + A\vartheta\left(0.41 + \Lh + 1.94 + 0.12\right)\\
 &\le \frac{3\Lh^2 + 9.81\Lh + 7.82}{4\pi}\,.
\end{align*}

Since $\widehat{\gamma_1} < T_1 < \widehat{\gamma_2}$,
we have to treat $\phi(\widehat{\gamma_1})$ separately. We have
\[
\phi(\widehat{\gamma_1}) =
 \frac{T}{\widehat{\gamma_1}(T-\widehat{\gamma_1})} < 0.083\,,
\]
and thus
\[
\sum_{0 \le \gamma \le T} \phi(\gamma)
 \le \frac{3\Lh^2 + 9.81\Lh + 8.87}{4\pi}\,.
\]
Hence, we have shown that
\begin{equation}		\label{eq:ABC_inner_bound}
\sum_{0 < \gamma_2 \le \gamma_1} t(\gamma_1,\gamma_2)
 \le  \frac{3\log^2(\gamma_1/2\pi) + 
      9.81\log(\gamma_1/2\pi) + 8.87}{4\pi\gamma_1^2}\,.
\end{equation}
We now consider the case that $\gamma_2$ is negative, whence $0 < -\gamma_2 \le \gamma_1$.
We could use Lemma~\ref{lem:Lehman}, but we adopt a simpler approach that
gives the same leading term.\footnote{This is not surprising, since
we use Lemma~\ref{lem:recip_gamma}, whose proof 
depends on Lemma~\ref{lem:Lehman}.}

Assuming that $\gamma_2 < 0$, we have 
$D \ge \gamma_1|\gamma_2|(\gamma_1+|\gamma_2|) \ge \gamma_1^2|\gamma_2|$,
and the terms are bounded by
\[
t(\gamma_1,\gamma_2) \le \frac{1}{\gamma_1^2|\gamma_2|}\,.
\]
Summing over $\gamma_2$ satisfying $0 < -\gamma_2 \le \gamma_1$,
using Lemma~\ref{lem:recip_gamma},
gives the bound
\begin{equation}			\label{eq:D_inner_bound}
\sum_{-\gamma_1 \le \gamma_2 < 0}t(\gamma_1,\gamma_2) \le
 \frac{\log^2(\gamma_1/2\pi)}{4\pi\gamma_1^2}\,.
\end{equation}

We now combine the results for positive and negative $\gamma_2$.
Adding the bounds~\eqref{eq:ABC_inner_bound}
and~\eqref{eq:D_inner_bound} gives
\begin{equation}                \label{eq:ABCD_inner_bound}
\sum_{-\gamma_1 \le \gamma_2 \le \gamma_1}\!\! t(\gamma_1,\gamma_2)
 \le  \frac{\log^2(\gamma_1/2\pi) + 
      2.46\log(\gamma_1/2\pi) + 2.22}{\pi\gamma_1^2}\,.
\end{equation}
Finally, we sum~\eqref{eq:ABCD_inner_bound} over all $\gamma_1 > T$
and use Lemmas~\ref{lem:sum0bd}--\ref{lem:sum2bd},
giving

\begin{align}					\nonumber
\sum_{\gamma_1 > T,\; |\gamma_2| \le \gamma_1}\!\! t(\gamma_1,\gamma_2)
 &\le \frac{(L^3-1.39L^2) + 2.46(L^2-L) + 2.22L}{2\pi^2 T}\\
						\label{eq:ABCD_outer_bound}
 &\le 
  \frac{L^3 + 1.1L^2}{2\pi^2 T}\,.
\end{align}

Allowing a factor of~$4$ for symmetry, and a
factor of~$5$ to allow for the numerator in~\eqref{eq:c6},
the tail bound $5E(T)$
is $20$ times the bound~\eqref{eq:ABCD_outer_bound}, so
\begin{equation}			\label{eq:tailbound}
5E(T) \le
\frac{10L^3 + 11L^2}{\pi^2 T}\,,
\end{equation}
which proves the theorem.
\end{proof}
It is possible to avoid the use of Lemma~\ref{lem:Lehman-combo} in the proof of Theorem~\ref{thm:main}, by summing the tail terms in a different order, so that the terms in the inner sums are monotonic decreasing and Lemma~\ref{lem:Lehman} applies. However, the resulting integrals are more difficult to bound than those occurring in our proof of Theorem~\ref{thm:main}. Both methods give the same leading term.
\begin{corollary}			\label{cor:B_lower_upper_bounds}
With the notation of Theorem~$\ref{thm:main}$, 
$0.8520 \le B \le 0.8603$.
\end{corollary}
\begin{proof}
The bounds on $B$ follow 
from Theorem~\ref{thm:main} by taking
$T = 260877$  
and evaluating the finite double sum,
which requires the first $4\cdot 10^5$ nontrivial zeros of $\zeta(s)$.
The evaluation, using interval arithmetic,
shows that the finite sum is in the interval
$[0.852089,0.852098]$, so the 
lower bound $0.8520$ stated in the corollary is correct.
The tail bound~\eqref{eq:tailbound} is $\le 0.008199$, 
and $0.852098 + 0.008199 = 0.860297$.
This proves the stated upper bound. 
\end{proof}

\begin{remark}				\label{rem:sharper_bounds}
Since the proof of Corollary~\ref{cor:B_lower_upper_bounds} uses
$T = 260877$, but Theorem~\ref{thm:main} and
Lemma~\ref{lem:sum2bd} assume only that $T \ge 100$,
it is natural to ask if the bounds can be improved if we assume that
$T$ is sufficiently large.  This is indeed the case. 
For $T \ge 80000$, the bound~\eqref{eq:ABCD_outer_bound}
can be improved to $(L^3+0.4L^2)/(2\pi^2 T)$, and it follows that 
the upper bound in Corollary~\ref{cor:B_lower_upper_bounds} can be improved
to $B \le 0.8599$.
The coefficient of $L^2$ in
the bound~\eqref{eq:ABCD_outer_bound} can be replaced by
$c(T) = 4 - 3\log 2 - \frac52\log\pi + \pi A + O(1/L) \le -0.06 + O(1/L)$,
and a bound on the $O(1/L)$ term shows that $c(T) \le 0$
for $T \ge 10^{42}$. 
The coefficient of $L^3$ is, however, the best that can be attained
by our method. 
\end{remark}

\section{Lower bound on $I(X)$}\label{brutus}
Stechkin and Popov \cite[Thm.\ 7]{Sketch} showed that, if RH were false, then $\liminf_{X\to\infty} I(X)/X^{2} = \infty $. Given this, we may as well assume RH in this section.
Stechkin and Popov \cite[Thm.\ 6]{Sketch} showed that we have for $X$ large enough
\begin{equation}\label{buzzard}
\int\limits_{X}^{2X}\left|\psi(u)-u\right|\dif u > \frac{X^{\frac{3}{2}}}{200},
\end{equation}
which by Cauchy--Schwarz leads immediately to $I(X)/X^{2} \geq (40\,000)^{-1}$.
The bound in (\ref{buzzard}) follows from showing under the same assumptions that
\begin{equation}\label{coaster}
H(X):=\int\limits_{X-\frac{\log 2}{2}}^{X+\frac{\log 2}{2}}\left|\sum\limits_{n\neq 0}\frac{\exp(i\gamma_n t)}{\rho_n}\right| \dif t > \frac{X^{\frac{3}{2}}}{200},
\end{equation}
where, throughout this section only, for $k\geq 1$ we define $\gamma_{k}$ (resp.\ $\gamma_{-k}$) to be the ordinate of the  $k$th non-trivial zero of $\zeta(s)$, above (resp.\ below) the real axis. We interpret the sum in (\ref{coaster}), which is not absolutely convergent, as 
$$\lim_{N\to\infty} \sum_{n=1}^{N}\left(\frac{ \exp(i \gamma_{n} t)}{\rho_{n}} + \frac{ \exp(i \gamma_{-n} t)}{\rho_{-n}}\right).$$

The key result we need is the following.
\begin{lemma}\label{lem:rho}
Let $g(z)$ be such that $g(0)=1$ and
$$
\delta=\frac{1}{\rho_1}-\sum\limits_{n\geq 2}\left|\frac{g(\gamma_n-\gamma_1)}{\rho_n}\right|-\sum\limits_{n\geq 1}\left|\frac{g(-\gamma_n-\gamma_1)}{\rho_n}\right|
$$
exists and is finite. Additionally, assume that
$$
\widehat{g}(y)=\frac{1}{2\pi}\int\limits_{\mathbb{R}} g(z)\exp(-i z y)\dif z
$$
exists and is supported on $[-\frac{1}{2}\log2,\frac{1}{2}\log2]$.
Then we have
$$
|H(X)|\geq\frac{\delta}{\max\limits_{y\in\mathbb{R}}\widehat{g}(y)}.
$$
\end{lemma} 
\begin{proof}
This follows from displays (15.4) to (17.4) of \cite[Sec.\ 4]{Sketch}.
\end{proof}

\begin{lemma}\label{lem:g}
Let $\alpha=\frac{\log 2}{6}$ and $\lambda>0$. Define
$$
g(z)=\left(\frac{\sin(\alpha z)}{\alpha z}\right)^3
 \left(1-\frac{z}{\lambda}\right)
$$
and
$$
\widehat{g}(y)=\frac{1}{2\pi}\int\limits_\mathbb{R}g(z)\exp(-i z y)\dif z.
$$
Then $g(0)=1$ and $\widehat{g}(y)$ is supported on $[-\frac{1}{2}\log 2,\frac{1}{2}\log2]$. Furthermore, for real $y$, $|\widehat{g}(y)|$ attains its maximum of $\frac{9}{4\log 2}$at $y=0$.
\end{lemma}
We note that Stechkin and Popov used the fourth power of the sinc function in place of our cube. Almost certainly better choices of the function $g(z)$ are possible: we leave this to future researchers, in the hope that they can thereby improve the lower bound in Theorem \ref{chalk}.
\begin{lemma}\label{lem:sigma}
Let $g$ be as defined in Lemma $\ref{lem:g}$. For $T>\max(\gamma_1+\lambda,2\pi\e)$ not the ordinate of a zero of $\zeta$ set
$$
\delta_{T,\lambda}=\sum\limits_{\gamma>T} \frac{|g(\gamma-\gamma_1)|+|g(-\gamma-\gamma_1)|}{\rho}.
$$
Then
$$
\delta_{T,\lambda}\le \int\limits_T^\infty h_{\lambda}(t)\log\frac{t}{2\pi}\dif t + 0.56 h_{\lambda}(T)\log T +0.28\int\limits_T^\infty\frac{h_{\lambda}(t)}{t}\dif t
$$
where
$$
h_{\lambda}(t)=\frac{t-{\lambda}-\gamma_1}{t(\alpha(t-\gamma_1))^3}+\frac{t+\lambda+\gamma_1}{t(\alpha(t+\gamma_1))^3}.
$$
\end{lemma}
\begin{proof}
This is a straightforward application of Corollary \ref{cor:improvedPTbound} and Lemma \ref{lem:Lehman}.
\end{proof}

\begin{corollary}
Let $\delta_{T,\lambda}$ be as in Lemma $\ref{lem:sigma}$,
with $T=446\,000$ and $\lambda=10.876$. Then
$$
\delta_{T,\lambda}\le 3.5\cdot 10^{-9}.
$$
\end{corollary}

We can now compute the contribution to $\delta$ from the $721\,913$ 
nontrivial zeros with imaginary part less than $446\,000$,
using $\lambda=10.876$. We find

$$
\frac{1}{|\rho_1|}-\sum\limits_{n=2}^{721\,913}\frac{g(\gamma_n-\gamma_1)}{\rho_n}-\sum\limits_{n=1}^{721\,913}\frac{g(-\gamma_n-\gamma_1)}{\rho_n}\geq 4.428\,225\,55\cdot 10^{-2},
$$
so we have
$
\delta\geq 0.044\,282\,252.
$

Appealing to Lemmas \ref{lem:rho} and \ref{lem:g} we can now claim
$$
|H(X)|\geq 0.044\,282\,252\frac{4\log 2}{9}\geq0.013\,641\,83,
$$
and the lower bound of Theorem \ref{chalk} results.

\section{Non-convergence of $I(X)/X^{2}$}\label{parsnip}

Our aim now is to show that $I(X)/X^2$ does not tend to a limit
as $X\to\infty$. It is more convenient to work with
\begin{equation}			\label{eq:mvJ}
J(X) := \int_0^X (\psi(x)-x)^2\dif x,
\end{equation}
and deduce results for 
$I(X)$.		
In Theorems~\ref{th:limsup} and \ref{th:liminf} 
we show that there exist effectively computable constants $c_1$ and
$c_2$, satisfying $c_1 < c_2$, such that
\[
\limsup_{X\to\infty}\frac{2}{X^2}\,J(X) \ge c_2, \quad  \liminf_{X\to\infty}\frac{2}{X^2}\,J(X) \le c_1.
\]
Hence $J(X)/X^2$ cannot tend to a limit as $X\to\infty$.
In Theorem~\ref{th:nolimit} we deduce
that $I(X)/X^2$ cannot tend to a limit $X\to\infty$.

\subsection{Some constants}		\label{sec:constants}
In sums over zeros,
each zero $\rho$ is counted according to its
multiplicity $m_\rho$. 
More precisely, a term involving $\rho$ is given a weight $m_\rho$.
In double sums, 
a term involving $\rho_1$ and $\rho_2$ is given a weight
$m_{\rho_1}m_{\rho_2}$.

We now define three real constants that are needed later.
First, a constant that appears in \cite[Thm.\ 13.6 and Ex.\ 13.1.1.3]{MV}
and our Theorem \ref{th:liminf}:
\begin{equation}			\label{eq:c1}
c_1 := 	\sum_{\rho} \frac{m_\rho}{|\rho|^2} \approx 0.046.
\end{equation}
Second, we define a constant that occurs in Theorem \ref{th:limsup}:
\begin{equation}			\label{eq:c2}
c_2 := \sum_{\rho_1,
\rho_2}\frac{2}{\rho_1\overline{\rho_2}(1+\rho_1+\overline{\rho_2})}
 \approx 0.104\,. 
\end{equation}
Observe that, assuming RH, 
the ``diagonal terms'' (i.e.\ those with $\rho_1=\rho_2$)
in \eqref{eq:c2} sum to $c_1$. 

Third, a constant that will be used in \S\ref{sec:c2_lower_bound}:
\begin{equation}			\label{eq:c3}
c_3 := \sum_{\gamma > 0}\frac{1}{\gamma^2} \le 0.023\;105,
\end{equation}
where this estimate has been computed to high accuracy previously (see, e.g.\ \cite{Saouter}).
We can replicate this result by summing numerically over zeros
below 
$3.72146\cdot 10^8$ 
and using Lemma \ref{lem:sum0bd} for the tail.

\subsection{The limsup result}		\label{sec:limsup}

We use the explicit formula
 for $\psi(x)$ (see, e.g., \cite[Thm.\ 12.5]{MV}) 
in the form
\[
\psi(x) - x = -\sum_{|\gamma| \le T} \frac{x^\rho}{\rho} + 
 O\left(\frac{x\log^{2} x}{T}\right)
\]
for $T \ge T_0$, $x \ge X_0$, and $x \ge T$.

\begin{theorem}
\label{th:limsup}
With $J(X)$ as in \eqref{eq:mvJ} and $c_2$ as in \eqref{eq:c2},
\[
\limsup_{X\to\infty} \frac{2J(X)}{X^2} \ge c_2.
\]
\end{theorem}
\begin{proof}
Fix some small $\varepsilon > 0$.
We can assume RH, since otherwise $J(X)/X^2$ is unbounded.
Proceeding as in the proof of ~\cite[Thm.\ 13.5]{MV}, but with the
integral over $[T,X]$ instead of $[X,2X]$, and using the
Cauchy--Schwartz inequality for the error term, we obtain
\[
\int_T^X (\psi(x)-x)^2\dif x = 
 \int_T^X\! \sum_{|\gamma_1| \le T,\, |\gamma_2| \le T} 
 \frac{x^{1+i(\gamma_1-\gamma_2)}}{\rho_1 \overline{\rho_2}}
 \dif x
 + O\left(\frac{X^{5/2}\log ^{2}X}{T}\right),
\]
provided $X \ge T \ge \max(T_0,X_0)$.
We also have, from \cite[Thm.\ 13.5]{MV},
\[
\int_0^T (\psi(x)-x)^2\dif x \ll T^2.
\]
Thus
\begin{align*}
\int_0^X (\psi(x)-x)^2\dif x = 
 \int_T^X\! &\sum_{|\gamma_1| \le T,\, |\gamma_2| \le T} 
 \frac{x^{1+i(\gamma_1-\gamma_2)}}{\rho_1 \overline{\rho_2}}
 \dif x\\
 & + O\left(T^2+X^{5/2}(\log X)^2/T\right).
\end{align*}
Now, from \cite[(13.16)]{MV},
$\displaystyle
\;
\sum_{\rho_1,\rho_2}
\left|
\frac{1}{\rho_1\overline{\rho_2}(2+i(\gamma_1-\gamma_2)}
\right|
\ll 1
$.\\
Thus, if we exchange the order of integration and summation (valid
since the sum is finite), and normalise by $X^2$, we obtain
\[
\frac{J(X)}{X^2}
 = \! \sum_{|\gamma_1| \le T,\, |\gamma_2| \le T}
 \frac{X^{i(\gamma_1-\gamma_2)}}{\rho_1 \overline{\rho_2}
  (2+i(\gamma_1-\gamma_2))}
  + O\left(\frac{T^2}{X^2} + \frac{X^{1/2}\log^{2} X}{T}\right).
\]
Choosing $T = X^{5/6}$, and assuming that $X \ge X_0^{6/5}$
so $T \ge X_0$,
the error term becomes $O(X^{-1/3}(\log X)^2)$.
Now, \hbox{choosing}
$X \ge \log^{6}(1/\varepsilon)/\varepsilon^3$, the error
term is $O(\varepsilon)$.
To summarise, we obtain error $O(\varepsilon)$ provided that
$T = X^{5/6}$ and $X \ge X_1$, where
$X_1 = \max(X_0^{6/5}, T_0^{6/5}, \log^{6}(1/\varepsilon)/\varepsilon^3)$.

We shall need another parameter
$Y = \log^{3}(1/\varepsilon)/\varepsilon$. Note that, by the conditions on
$T$ and $X$, we necessarily have $Y \le T$ for 
$\varepsilon\in (0, 1/e)$, since
$T = X^{5/6} \ge \log(1/\varepsilon)^{5}/\varepsilon^{5/2}
 \ge \log^{3}(1/\varepsilon)/\varepsilon = Y$.

It remains to consider the main sum over pairs 
$(1/2+i\gamma_1, 1/2-i\gamma_2)$ of zeros with
$|\gamma_1|, |\gamma_2| \le T$. Observe that the sum is
real, as we can see by grouping the term for 
$(1/2+i\gamma_1,1/2-i\gamma_2)$
with the conjugate term for $(1/2-i\gamma_1, 1/2+i\gamma_2)$.
Using Dirichlet's theorem~\cite[\S8.2]{Titch}, 
we can find some $t \ge \log X_1$, such that
$|\{t\gamma/(2\pi)\}| \le \varepsilon$
for all zeros $1/2 + i\gamma$ with $0 < \gamma \le Y$,
where $Y \le T$ is as above.\footnote{Here
$\{x\}$ 
denotes the fractional part of~$x$.}
Set $X = \exp(t)$. Then, for all the $(\gamma_1,\gamma_2)$
occurring in the main sum with $\max(|\gamma_1|,|\gamma_2|) \le Y$,
we have
$X^{i(\gamma_1-\gamma_2)} = 1 + O(\varepsilon)$.
Hence, for this choice of $X$, we have
\begin{equation*}			
\frac{J(X)}{X^2}
 = \! \sum_{|\gamma_1| \le Y,\, |\gamma_2| \le Y}
 \frac{1}{\rho_1 \overline{\rho_2}
  (2+i(\gamma_1-\gamma_2))} + R(Y) + O(\varepsilon),
\end{equation*}
where
\begin{equation*}	
|R(Y)| \le \sum_{\max(|\gamma_1|,|\gamma_2|) > Y}
 \left| \frac{1}{\rho_1 \overline{\rho_2}
  	(2+i(\gamma_1-\gamma_2))} \right|
	\ll \frac{\log^{3} Y}{Y}
\end{equation*}
is the tail of an absolutely convergent
double sum, 
see (\ref{eq:ST}) and \cite[p.\ 424]{MV}.
Thus, with our choice $Y = \log^{3}(1/\varepsilon)/\varepsilon$,
we have $R(Y) = O(\varepsilon)$.

Recalling the definition of the constant $c_2$ in \eqref{eq:c2},
we have shown
that, for any sufficiently small $\varepsilon > 0$,
there exists $X = X(\varepsilon)$ such that
\begin{equation}			\label{eq:c2eps}
\frac{2J(X)}{X^2} \ge c_2 - O(\varepsilon).
\end{equation}
Since $\varepsilon$ can be arbitrarily small, this proves the result.
\end{proof}

\begin{remark}
The least $X$ satisfying \eqref{eq:c2eps}
may be bounded using \hbox{\cite[(8.2.1)]{Titch}}. 
The result is 
doubly exponential in $1/\varepsilon$. More precisely,
\begin{equation*}		
X(\varepsilon) \le \exp(\exp((1/\varepsilon)^{1+o(1)}))
 \text{ as $\varepsilon \to 0$}.
\end{equation*}
\end{remark}

\subsection{A lower bound on $c_2$}	\label{sec:c2_lower_bound}
The constants $c_1$ and $c_2$ are of little interest, so far
as the theory of $\psi(x)$ goes, if RH is false.
Hence, we assume RH.
In Corollary~\ref{cor:c1c2} we show that $c_1 < c_2$. Although 
computations of $c_2$ suggest this, they do not provide
a proof unless they come with a
(possibly one-sided) error bound. Here we show how rigorous lower
bounds on $c_2$ can be computed. This provides a way of proving rigorously,
without extensive computation, that $c_1 < c_2$.

First we extract the real part of the expression~\eqref{eq:c2}.
This leads to sharper bounds on the terms than if we included the imaginary
parts, which must ultimately cancel.

\begin{lemma}			\label{lem:c2_real_form}
Assume RH. If $c_2$ is defined by \eqref{eq:c2}, then
\[
c_2 = \sum_{\gamma_1>0,\;\gamma_2} T(\gamma_1,\gamma_2),
\]
where
\begin{equation}		\label{eq:c2_real_form}
T(\gamma_1,\gamma_2) = \frac{2(1+6\gamma_1\gamma_2-\gamma_1^2-\gamma_2^2)}
{(\frac14 + \gamma_1^2)(\frac14+\gamma_2^2)(4+(\gamma_1-\gamma_2)^2)}
\,.
\end{equation}
\end{lemma}
\begin{proof}
We expand \eqref{eq:c2},
using $\rho_j = \frac12 + i\gamma_j$ (this is where RH is required),
omit the imaginary parts since the final result is real,
and use symmetry to reduce to the case $\gamma_1 > 0$ (so in the 
resulting sum, $\gamma_1$ is positive but $\gamma_2$ may have either sign).
\end{proof}

Lemma \ref{lem:positive_region}
gives a region in which the terms occurring
in~\eqref{eq:c2_real_form} are positive. 
\begin{lemma}			\label{lem:positive_region}
If $T(\gamma_1,\gamma_2)$ is as in \eqref{eq:c2_real_form}, and
$\gamma_2/\gamma_1 \in [3-\sqrt{8},3+\sqrt{8}]$, then
$T(\gamma_1,\gamma_2) > 0$.
\end{lemma}
\begin{proof}
Since the denominator of $T(\gamma_1,\gamma_2)$ is positive, it is sufficient
to consider the numerator, which we write as
$2P(\gamma_1,\gamma_2)$, where
\[
P(x,y) = 1 + 6xy - x^2 - y^2.
\]
Let $r = y/x$, so $P(x,y) = 1 - (r^2 - 6r + 1)x^2$.
Now $r^2-6r+1 = (r-3)^2-8$ vanishes at $r = 3\pm\sqrt{8}$, and is 
negative iff $r\in (3-\sqrt{8},3+\sqrt{8})$.
Thus $P(x,y)$ is positive for $r\in [3-\sqrt{8},3+\sqrt{8}]$.
Taking $x=\gamma_1, y=\gamma_2$ proves the lemma.
\end{proof}

Define
\[
S(Y) = \sum_{{\;\;\;0 < \gamma_1 \le Y}\atop{-Y \le \gamma_2 \le Y}}
 T(\gamma_1,\gamma_2).
\]
Then $c_2 = \lim_{Y\to\infty} S(Y)$.
Clearly $S(Y)$ is constant between ordinates of nontrivial zeros of $\zeta(s)$,
and has jumps 
\[
J(\gamma) = \lim_{\varepsilon\to 0}(S(\gamma+\varepsilon)-S(\gamma-\varepsilon))
\]
at positive ordinates $\gamma$ of zeros of $\zeta(s)$.
We shall show that all these jumps are positive, so $S(Y)$ is monotonic
non-decreasing, and $c_2 > S(Y)$ for all $Y>0$.
This allows us to prove that $c_2 > c_1$
by computing $S(Y)$ for sufficiently large~$Y$
(see Corollary \ref{cor:c1c2}).

If $\gamma>0$ is the ordinate of a simple zero\footnote{For simplicity we
assume here that all zeros of $\zeta(s)$ are simple, but one can
modify the proofs in an obvious way to account for multiple zeros,
if they exist.} of $\zeta(s)$, then
\begin{align}
\nonumber
J(\gamma) &= \sum_{0<\gamma_1\le\gamma}T(\gamma_1,\gamma)
	   +\!\!\sum_{0<\gamma_1\le\gamma}T(\gamma_1,-\gamma)
	   +\!\!\!\sum_{-\gamma<\gamma_2<\gamma}T(\gamma,\gamma_2)\\
							\label{eq:lastsum}
	&= T(\gamma,\gamma) 	   + T(\gamma,-\gamma)
	   + 2\!\!\!\!\sum_{-\gamma<\gamma_2<\gamma} T(\gamma,\gamma_2)\,.
\end{align}
This may be seen by drawing a rectangle with vertices at
$(0,\gamma)$, $(\gamma,\gamma)$, $(\gamma,-\gamma)$, $(0,-\gamma)$,
following the north, east and south edges, and using the
symmetry \hbox{$T(x,y)=T(y,x)$}.

To show that $J(\gamma)>0$, we split the last sum in~\eqref{eq:lastsum}
into three pieces, 
$A := (-\gamma,0]$,
$B := (0, (3-\sqrt{8})\gamma)$, and
$C := [(3-\sqrt{8})\gamma,\gamma)$.
This gives
\begin{align*}
J(\gamma) =&\;\; T(\gamma,\gamma) + T(\gamma,-\gamma)\\
	&+ 2\sum_{\gamma_2\in A}T(\gamma,\gamma_2)
	 + 2\sum_{\gamma_2\in B}T(\gamma,\gamma_2)
	 + 2\sum_{\gamma_2\in C}T(\gamma,\gamma_2).
\end{align*}
By Lemma~\ref{lem:positive_region}, the sum with $\gamma_2\in C$
consists only of positive
terms, so
\begin{equation}		\label{eq:two_sums}
J(\gamma) \ge T(\gamma,\gamma) + T(\gamma,-\gamma)
	+ 2\sum_{\gamma_2\in A}T(\gamma,\gamma_2)
	+ 2\sum_{\gamma_2\in B}T(\gamma,\gamma_2).
\end{equation}

We now show that the diagonal term $T(\gamma,\gamma)$ in \eqref{eq:two_sums}
is positive, and sufficiently large to dominate 
the anti-diagonal term $T(\gamma,-\gamma)$ and the sums over $A$ and $B$.

\begin{lemma}[diagonal term]							\label{lem:diag}
							We have
$T(\gamma,\gamma) \ge 1.99/\gamma^2\,.
$
\end{lemma}
\begin{proof}
Since $\gamma>0$ is the ordinate of a nontrivial zero of $\zeta(s)$,
we have $\gamma > 14$.
Thus, using \eqref{eq:c2_real_form}, we have 
$T(\gamma,\gamma) = {2}/{(\frac14+\gamma^2)} > {1.99}/{\gamma^2}.$
\end{proof}

\begin{lemma}[anti-diagonal term and interval $A$]
					\label{lem:A}
If $c_3$ is as in \eqref{eq:c3}, then
\[\frac{|T(\gamma,-\gamma)|}{2} + 
 \sum_{-\gamma < \gamma_2 < 0} |T(\gamma,\gamma_2)|
 \le \frac{16c_3}{\gamma^2}
 < \frac{0.37}{\gamma^2}\,.
\]
\end{lemma}
\begin{proof}
Write \eqref{eq:c2_real_form} as $T(\gamma,\gamma_2) = N/D$,
where the numerator is
\begin{equation}
N = 2(1+6\gamma\gamma_2-\gamma^2-\gamma_2^2),	\label{eq:N}
\end{equation}
and the denominator is
\begin{equation}				\label{eq:D}
D = (\textstyle\frac14+\gamma^2)(\frac14+\gamma_2^2)
	(4+(\gamma-\gamma_2)^2)
	> \gamma^2\gamma_2^2(\gamma-\gamma_2)^2.
\end{equation}
Thus, $N/2 = 1 - (r^2 -6r + 1)\gamma^2$, where 
$r = \gamma_2/\gamma$.
Now $r^2-6r+1 \in [1,8]$ for $r\in [-1,0]$.
Thus $N/2 \in [1-8\gamma^2,1-\gamma^2]$, and $|N| < 16\gamma^2$.

For the denominator, we have $D > \gamma^4\gamma_2^2(1-r)^2
\in [\gamma^4\gamma_2^2, 4\gamma^4\gamma_2^2]$,
so $D > \gamma^4\gamma_2^2$.
Combining the inequalities for $N$ and $D$ gives
\[
|T(\gamma,\gamma_2)| < \frac{16}{\gamma^2\gamma_2^2}\,.
\]
Now, summing over $\gamma_2 < 0$, 
and recalling the definition of $c_3$ in \eqref{eq:c3}, gives the result.
\end{proof}

\begin{lemma}[interval $B$]		\label{lem:B} We have
\[\sum_{0 < \gamma_2 < (3-\sqrt{8})\gamma} |T(\gamma,\gamma_2)|
 \le \frac{(3+\sqrt{8})c_3}{2\gamma^2}
 < \frac{0.068}{\gamma^2}\,.
\]
\end{lemma}
\begin{proof}
As in the proof of Lemma~\ref{lem:A},
write \eqref{eq:c2_real_form} as $T(\gamma,\gamma_2) = N/D$,
where $N$ and $D$ are as in \eqref{eq:N}--\eqref{eq:D}.
Now
$\gamma_2/\gamma < 3-\sqrt{8}$,
so $1-\gamma_2/\gamma > \sqrt{8}-2$,
and 
$(\gamma-\gamma_2)^2 > 4(3-\sqrt{8})\gamma^2$.  This gives
\[
D > 4(3-\sqrt{8})\gamma^4\gamma_2^2. 
\]
Also, $N/2 = 1 - (r^2 -6r + 1)\gamma^2$, where 
$r = \gamma_2/\gamma \in [0,3-\sqrt{8}]$.
Thus $0 \le r^2-6r+1 \le 1$ and $|N| \le 2\gamma^2$.
The inequalities for $D$ and~$N$ give
\[
|T(\gamma,\gamma_2)| < \frac{2\gamma^2}{4(3-\sqrt{8})\gamma^4\gamma_2^2}
 = \frac{3+\sqrt{8}}{2\gamma^2\gamma_2^2}\,.
\]
Now, summing over $\gamma_2 > 0$ gives the result.
\end{proof}

\begin{lemma}				\label{lem:S}
$S(Y)$ is monotonic non-decreasing for $Y\in [0,\infty)$,
with jumps of at least $1.11/\gamma^2$ at ordinates $\gamma > 0$
of $\zeta(s)$.
\end{lemma}
\begin{proof}
Using the inequality \eqref{eq:two_sums}
and Lemmas~\ref{lem:diag}--\ref{lem:B}, we have
\[
J(\gamma) \ge \frac{1.99 - 2\cdot 0.37 - 2\cdot 0.068}{\gamma^2}
	  > \frac{1.11}{\gamma^2}\,.
\]
Thus, $S(Y)$ has positive jumps at ordinates $\gamma > 0$ of zeros
of $\zeta(s)$, and is constant between these ordinates.
\end{proof}

\pagebreak[3]

\begin{corollary}			\label{cor:monotone}
Assume RH. For all $Y>0$, we have $c_2 > S(Y)$.
\end{corollary}
\begin{proof}
This follows as $S(Y)$ is monotonic non-decreasing with limit $c_2$,
and has positive jumps at arbitrarily large $Y$.
\end{proof}

\begin{corollary}			\label{cor:c1c2}
Assume RH. Then $c_1 < c_2$.
\end{corollary}
\begin{proof}
Take $Y=70$ in Corollary \ref{cor:monotone}.
Computing $S(70)$, which involves a double sum over
first $17$ nontrivial zeros in the upper half-plane,
gives a lower bound
$c_2 > S(70) > 0.0466$. Since $c_1 < 0.0462$, the result follows.
\end{proof}

\begin{remark}
RH is probably not necessary for Corollary~\ref{cor:c1c2}.
Any exceptional zeros off the critical line must have large
height, and consequently they would make little difference to the
numerical values of $c_1$ and $c_2$.
\end{remark}

\begin{remark}				\label{rem:lowerbd}
Taking $Y = 74\,920.83$ in Corollary~\ref{cor:monotone},
and using the first $10^5$ zeros of $\zeta(s)$, 
we obtain
\[c_2 > S(Y) > 0.104004 \text{ and } c_2 - c_1 > 0.0578\,.\]
This is much stronger than 
the bound used in the proof of Corollary~\ref{cor:c1c2},
though at the expense of more computation.
Our best estimate, using an integral approximation
for the higher zeros, is $c_2 \approx 0.10446$\,.
\end{remark}

\subsection{Non-existence of a limit}		\label{sec:nolimit}
First we prove a result analogous to
Theorem~\ref{th:limsup}, but with $\limsup$ replaced by $\liminf$.
Then we deduce that neither $I(X)/X^2$ nor
$J(X)/X^2$ has a limit as $X\to\infty$.

\begin{theorem}					\label{th:liminf}
Assume RH. With $J(X)$ as in \eqref{eq:mvJ} and $c_1$ as in \eqref{eq:c1},
\[
\liminf_{X\to\infty} \frac{2J(X)}{X^2} \le c_1.
\]
\end{theorem}
\begin{proof}
Define
\begin{align*}
F(X) :=& \int_1^X (\psi(x)-x)^2\dif x = J(X)-J(1), \text{ and}\\
G(X) :=& \int_1^X (\psi(x)-x)^2\,\frac{dx}{x^2} \sim c_1\log X.
\end{align*}
Here the asymptotic result is given in \cite[Ex.~13.1.1.3]{MV}, which follows
from \hbox{\cite[Thm.\ 13.6]{MV}} after
a change of variables $x = \exp(u)$.
Using integration by parts, we obtain
\[
G(X) = \frac{F(X)}{X^2} + 2\int_1^XF(x)\,\frac{dx}{x^3}\,.
\]
Now $F(X) \ll X^2$, so 
\[
2\int_1^XF(x)\,\frac{dx}{x^3} 
 \sim G(X)
 \sim c_1\log X \text{ as $X \to \infty$}.
\]
Dividing by $2\log X$ gives
\begin{equation}		\label{eq:meanratio}
\int_1^X\frac{F(x)}{x^2}\,\frac{dx}{x} \left/
 \int_1^X\frac{dx}{x}\right. \sim \frac{c_1}{2} \text{ as $X \to \infty$}.
\end{equation}
Now, if $F(x)/x^2 \ge c_1/2+\varepsilon$ for some positive $\varepsilon$
and all sufficiently large $x$,
we get a contradiction to \eqref{eq:meanratio}. Thus, letting
$\varepsilon\to 0$, we obtain the result.
\end{proof}
\noindent
\begin{corollary}		\label{cor:nolimitforJ}
With $J(X)$ as in \eqref{eq:mvJ},
$\displaystyle\lim_{X\to\infty} \frac{J(X)}{X^2}$ does not exist.
\end{corollary}
\begin{proof}
The result holds if RH is false.  Hence, assume RH.
{From} Corollary~\ref{cor:c1c2}, 
$c_1 < c_2$, so the result is implied by Theorems~\ref{th:limsup} 
and~\ref{th:liminf}.
\end{proof}
We conclude by showing the non-existence of $\lim_{X\rightarrow\infty} I(X)X^{-2}$, thereby proving Theorem \ref{th:nolimit}.
Suppose, on the contrary, that
the limit exists.
Now, from the definitions \eqref{turnip} and \eqref{eq:mvJ}, we have
\[
\frac{J(X)}{X^2} = \sum_{k=1}^\infty \frac{I(X/2^k)}{X^2}
	= \sum_{k=1}^\infty 4^{-k}\frac{I(X/2^k)}{(X/2^k)^2}\,,
\]
and the series converge since the $k$-th terms are $O(4^{-k})$.
Hence there \hbox{exists} $\lim_{X\to\infty}J(X)/X^2$, but this contradicts
Corollary~\ref{cor:nolimitforJ}. Thus, our original assumption is false,
and the result follows.


\begin{thebibliography}{99}

\bibitem{Backlund}
R. J. Backlund.
\"Uber die Nullstellen der Riemannschen Zetafunktion.
\emph{Acta Math.\ } 41:345--375, 1918.

\bibitem{Buethe}
J. B\"{u}the.
\newblock Estimating $\pi(x)$ and related functions under partial RH assumptions.
\newblock {\em Math. Comp.}, 85(301):2483--2498, 2016.

\bibitem{Cramer}
H.\ Cram\'{e}r.
\newblock Ein Mittelwertsatz in der Primzahltheorie.
\newblock {\em Math. Z.}, 12:147--153, 1922.

\bibitem{Saouter}
P. Demichel, Y. Saouter, and T. Trudgian.
\newblock A still sharper region where $\pi(x) - \textrm{li}(x)$ is positive.
\newblock{\em Math. Comp.}, 84(295):2433--2446, 2015.

\bibitem{Hassani}
M. Hassani.
\newblock Explicit approximation of the sums over the imaginary part of the non-trivial zeros of the Riemann zeta function.
\newblock {\em Appl. Math. E-Notes}, 16:109--116, 2016.

\bibitem{Lehman}
R.~S. Lehman.
\newblock On the difference $\pi(x) - \textrm{li}(x)$.
\newblock{\em Acta Arith.}, 11:397--410, 1966.

\bibitem{MV}
H. Montgomery and R.~C. Vaughan. \newblock {\em Multiplicative Number Theory. I. Classical Theory.}
\newblock Cambridge Studies in Advanced Mathematics, 97. Cambridge University Press, Cambridge, 2007.

\bibitem{Pintz3}
J. Pintz.
\newblock On the remainder term of the prime number formula VI. Ineffective mean value theorems.
\newblock {\em Studia Sci. Math. Hungar.}, 15:225--230, 1980.

\bibitem{Pintz2}
J. Pintz.
\newblock On the remainder term of the prime number formula and the zeros of Riemann's zeta-function.
\newblock {\em Number theory, Noordwijkerhout 1983}, Lecture Notes in Mathematics, 1068, Springer-Verlag, Berlin, 1984.


\bibitem{Pintz1}
J. Pintz.
\newblock On the mean value of the remainder term of the prime number formula. In  
\newblock {\em Elementary and Analytic Theory of Numbers (Warsaw, 1982)},  411--417, Banach Center Publ., 17, PWN, Warsaw, 1985.





\bibitem{PTZ}
D.~J. Platt and T.~S. Trudgian.
\newblock An improved explicit bound on \hbox{$|\zeta(\frac{1}{2} + it)|$}.
\newblock {\em J. Number Theory}, 147:842--851, 2015.


\bibitem{Sketch}
S.~B. Stechkin and A.~Yu. Popov.
\newblock Asymptotic distribution of prime numbers in the mean.
\newblock {\em Russian Math. Surveys}, 51(6):1025--1092, 1996.

\bibitem{Titch}
E.~C. Titchmarsh,
edited and with a preface by D.~R.~Heath-Brown.
\newblock {\em The Theory of the Riemann Zeta-Function},
2nd edition. Oxford Univ.\ Press, New York, 1986.

\bibitem{TrudgianS2}
T. S. Trudgian.
\newblock An improved upper bound for the argument of the Riemann zeta-function on the critical line, II.
\newblock {\em J. Number Theory}, 134:280--292, 2014.

\end{thebibliography}
\end{document}